\documentclass{amsart}

\usepackage{amsmath,amssymb,epic,graphicx,mathrsfs,enumerate,color}
\usepackage{amsfonts}
\usepackage{amsthm}
\usepackage{amssymb}
\usepackage{latexsym}
\usepackage{longtable}
\usepackage{epsfig}
\usepackage{amsmath}
\usepackage{hhline}
\usepackage{float}
\restylefloat{table}

\input xy
\xyoption{all}

\newtheorem{teor}{Theorem}

\newtheorem{prop}{Proposition}

\newtheorem{lemma}{Lemma}

\newcommand{\abs}[1]{\left\vert#1\right\vert}         

\begin{document}

\title[Elements pairwise generating the symmetric group]{On the maximal number 
of elements pairwise generating the symmetric group of even degree}

\author{Francesco Fumagalli} 
\address{Dipartimento di Matematica e Informatica 'Ulisse Dini', 
Viale Morgagni 67/A, 50134 Firenze, Italy}
\email{francesco.fumagalli@unifi.it}

\author{Martino Garonzi}
\address{Departamento de Matem\'atica, Universidade de Bras\'ilia, Campus 
Universit\'ario \\ Darcy Ribeiro, Bras\'ilia-DF, 70910-900, Brazil}
\email{mgaronzi@gmail.com}

\author{Attila Mar\'oti}
\address{Alfr\'ed R\'enyi Institute of Mathematics, Re\'altanoda utca 13-15, H-1053, 
Budapest, Hungary}
\email{maroti.attila@renyi.hu}

\thanks{The second author acknowledges the support of Funda\c{c}\~{a}o de Apoio \`a 
Pesquisa do Distrito Federal (FAPDF) - demanda espont\^{a}nea 03/2016, and of 
Conselho Nacional de Desenvolvimento Cient\'ifico e Tecnol\'ogico (CNPq) - Grant 
numbers 302134/2018-2, 422202/2018-5. The work of the third author on the project 
leading to this application has
received funding from the European Research Council (ERC) under the
European Union's Horizon 2020 research and innovation programme
(grant agreement No. 741420). He was also
supported by the National Research, Development and Innovation Office
(NKFIH) Grant No.~K132951, Grant No.~K115799, Grant No.~K138828.}

\date{}



\begin{abstract}
Let $G$ be the symmetric group of degree $n$. Let $\omega(G)$ be the maximal size of a subset $S$ of $G$ such that $\langle x,y \rangle = G$ whenever $x,y \in S$ and $x \neq y$ and let $\sigma(G)$ be the minimal size of a family of proper subgroups of $G$ whose union is $G$. We prove that both functions $\sigma(G)$ and $\omega(G)$ are asymptotically equal to $\frac{1}{2} \binom{n}{n/2}$ when $n$ is even. This, together with a result of S. Blackburn, implies that $\sigma(G)/\omega(G)$ tends to $1$ as $n \to \infty$. Moreover, we give a lower bound of $n/5$ on $\omega(G)$ which is independent of the classification of finite simple groups. We also calculate, for large enough $n$, the clique number of the graph defined as follows: the vertices are the elements of $G$ and two vertices $x,y$ are connected by an edge if $\langle x,y \rangle \geq A_n$.
\end{abstract}
\maketitle

\begin{center}
{\it To the memory of Carlo Casolo.}
\end{center}

\section{Introduction}

Let $G$ be a noncyclic finite group. J.H.E.~Cohn in \cite{cohn} defined $\sigma(G)$ to be the 
minimal number $k$ such that $G$ is the union of $k$ proper subgroups of $G$. 
This invariant has been studied by many authors. In particular, M.J.~Tomkinson \cite{tom} 
proved that if $G$ is solvable then $\sigma(G)=q+1$, where $q$ is the smallest order 
of a chief factor of $G$ with more than one complement. 
In the present work we concentrate on the symmetric group $S_n$ of degree $n$. 
In two papers \cite{Maroti,KNS} it was shown that $\sigma(S_n)=2^{n-1}$ for every 
odd integer $n \geq 3$. The determination of $\sigma(S_n)$ when $n$ is even seems a more difficult task. In case $n$ is divisible by $6$, E. Swartz \cite{Swartz} 
managed to give a formula for $\sigma(S_n)$. 
Apart from this case, the value of $\sigma(S_n)$ for $n$ even is only known for $n \leq 14$: see \cite{AAS,KNS,OppenheimSwartz}. 

Let $G$ be a finite group which can be generated by $2$ elements. A subset $S$ of 
$G$ is called a pairwise generating set if every subset of $S$ of size $2$ generates 
$G$. The maximal size of a pairwise generating set for $G$ is denoted by $\omega(G)$. This invariant was first introduced by M.W.~Liebeck and A.~Shalev in \cite{LS96}, 
where a general lower bound for $\omega(G)$ was given for $G$ a nonabelian finite 
simple group: they proved that $\omega(G) \geq c \cdot m(G)$ where $c$ is an absolute 
positive constant and $m(G)$ denotes the minimal index of a proper subgroup of $G$. 
After some initial results in \cite{Maroti}, S.R.~Blackburn \cite{B} proved that 
$\omega(S_n)=2^{n-1}$ provided that $n$ is odd and sufficiently large. 
Later, in her Ph.D thesis, L.~Stringer \cite{Stringer} studied the small odd values of $n$ and showed that $\omega(S_n)=2^{n-1}$ for every odd integer $n$ at least $17$ or belonging to $\{7,11,13\}$. Moreover, she showed that $\omega(S_5) < \sigma(S_5)$ and that $\omega(S_9) < \sigma(S_9)$ (see also \cite{KNS}); the problem of whether $\omega(S_{15})=\sigma(S_{15})$ or not is still open. 

An obvious connection between $\sigma(G)$ and $\omega(G)$ for any noncyclic finite group $G$ is that $\omega(G) \leq 
\sigma(G)$. Indeed, every proper subgroup of $G$ contains at most one element of any 
pairwise generating set for $G$. 

Our first result is the following. 

\begin{teor} \label{asympt}
If $n$ is even then $\sigma(S_n)$ and $\omega(S_n)$ are asymptotically equal to $\frac{1}{2} \binom{n}{n/2}$.
\end{teor}

This, together with S. Blackburn's result mentioned above, implies that the quotient $\sigma(S_n)/\omega(S_n)$ tends to $1$ as $n$ tends to infinity, without restrictions on the parity of $n$.


The idea of the proof of Theorem \ref{asympt} is to show that there exists a set of pairwise generating elements of $S_n$, consisting of $n$-cycles, one in each imprimitive maximal subgroup of $S_n$ with two blocks of imprimitivity. This gives the lower bound in the following chain of inequalities.
$$\frac{1}{2} \binom{n}{n/2} \leq \omega(S_n) \leq \sigma(S_n) \leq \frac{1}{2} \binom{n}{n/2} + \sum_{i=1}^{\lfloor n/3 \rfloor} \binom{n}{i}.$$
The upper bound is obtained noting that $S_n$ is covered by the imprimitive maximal subgroups with $2$ blocks and the intransitive maximal subgroups stabilizing sets of size at most $\lfloor n/3 \rfloor$. The result then follows by letting $n \to \infty$. 

Let $\Gamma_n$ be the graph whose vertices are the elements of $S_n$ which are products of exactly three disjoint cycles and there is an edge between two of them if they generate a transitive subgroup of $S_n$. The main combinatorial obstacle to determine $\omega(S_n)$ and/or $\sigma(S_n)$ is to determine the clique number of $\Gamma_n$.

Our proof of Theorem \ref{asympt} makes use of results about maximal primitive 
subgroups of $S_n$ (see Lemma \ref{cs}) that rely on the Classification of Finite Simple Groups (CFSG). 
However \cite{B} also depends on CFSG.

We remark that, apart from some symmetric groups, the only cases in which the precise value of 
$\omega$ is known are for groups of Fitting height at most $2$ \cite{LM}, for certain 
alternating groups \cite{B} and for certain linear groups \cite{BEGHM}.

If we allow the pairs of elements of $S_n$ to generate $S_n$ or $A_n$ (and $n$ is even), then we are able to determine the precise size of certain subsets as in our second theorem. 

\begin{teor} \label{omega2}
If $n$ is a large enough even integer, then the maximal size of a subset $X$ of $S_n$ with the property that $\langle x,y \rangle \geq A_n$ whenever $x,y$ are two distinct elements of $X$ is $\frac{1}{2} \binom{n}{n/2} + 2^{n-2}$ if $n/2$ is even and $2^{n-2}$ if $n/2$ is odd.
\end{teor}

At the heart of the proof of Theorems \ref{asympt} and \ref{omega2} is the Lov\'asz Local Lemma \cite{lov}. In this context, the Local Lemma was first used by S.R. Blackburn \cite{B} and 
elaborated on by L.~Stringer \cite{Stringer} in her Ph.D thesis.


Our last result does not depend on CFSG via a nice theorem of Eberhard and Virchow \cite{EV}. 

\begin{prop}
\label{CFSG}
Both $\omega(S_{n})$ and $\omega(A_{n})$ are at least $(1-o(1))n$.
\end{prop}

\section{The local lemma}

Given an event $E$ of a probability space, we denote by $P(E)$ its probability and by $\overline{E}$ its complement. As usual $e$ denotes the base of the natural logarithm. 


The following crucial result can be found in \cite{lov}. The formulation we use is taken from \cite[Corollary 5.1.2]{AS} (the ``symmetric case'').

\begin{teor}[Lov\'asz Local Lemma] \label{local}
Let $A_1,\ldots,A_n$ be events in an arbitrary probability space. Let $(V,E)$ be a directed graph, where $V=\{1,\ldots,n\}$, and assume that, for every $i \in V$, the event $A_i$ is mutually independent of the set of events $A_j$ such that $(i,j) \not \in E$. Let $d$ be the maximum valency of a vertex of the graph $(V,E)$. If for every $i\in V$
$$P(A_i) \leq \frac{1}{e(d+1)}$$
then $P( \bigcap_{i \in V} \overline{A_i}) > 0.$

\end{teor}

The mutual independence condition mentioned in the Lov\'asz Local Lemma means the following:
$$P(A_i\ |\ \bigcap_{j \in S} \overline{A_j}) = P(A_i) \hspace{.5cm} \forall i \in V, \hspace{.5cm} \forall S \subseteq \{j \in V\ :\ (i,j) \not \in E\}.$$


\section{Proof of Theorems \ref{asympt} and \ref{omega2}}

From now on let $n$ be a large even integer. Let also $G:=S_n$ be the symmetric group on $n$ letters. Let $A_n$ denote the alternating 
group on $n$ letters. We will prove both theorems using the same argument.

Let $\mathscr{M}^{(1)}$ be the family of maximal imprimitive subgroups of $G$ with $2$ blocks and let $\Pi^{(1)}$ be the set of $n$-cycles in $G$. Let $\mathscr{M}^{(2)}$ be the family of maximal subgroups of $G$ that are either imprimitive with $2$ blocks or intransitive of type $S_a \times S_b$ with $a$ and $b$ odd, $a \neq b$, $a+b=n$, and let $\Pi^{(2)}$ be the set of elements of $G$ that are either $n$-cycles or elements of cycle type $(a,b)$ with $a$ and $b$ odd, $a \neq b$, $a+b=n$.

Note that 
$$|\mathscr{M}^{(1)}| = \frac{1}{2} \binom{n}{n/2}, \hspace{1cm} |\mathscr{M}^{(2)}| = \left\{ \begin{array}{ll} \frac{1}{2} \binom{n}{n/2}+2^{n-2} & \mbox{if } n/2 \mbox{ is even,} \\
2^{n-2} & \mbox{if } n/2 \mbox{ is odd.} \end{array} \right.$$

Moreover $\mathscr{M}^{(2)}$ is a covering of $G$, meaning that $\bigcup_{M \in \mathscr{M}^{(2)}} M = G$. To see this, note that the elements whose cycle structure consists of cycles all of even length or of exactly two cycles of equal length are covered by the imprimitive maximal subgroups with $2$ blocks, while the elements that admit in the cycle decomposition a cycle of odd length less than $n/2$ are covered by some intransitive subgroup of type $S_a \times S_b$ where $a$, $b$ are odd and $a \neq b$.

Let $\mathscr{S}^{(1)}$ be the set of subsets of $\Omega=\{1,\ldots,n\}$ of size $n/2$ and containing $1$. Let $\mathscr{S}^{(2)}$ be the union of
$\mathscr{S}^{(1)}$ with the set of subsets of $\Omega$ of
odd size less than $n/2$. There is a natural bijection $\mathscr{S}^{(i)} \to \mathscr{M}^{(i)}$, $\Delta \mapsto M_{\Delta}$ for $i=1,2$. Specifically, if $|\Delta|=n/2$ then $M_{\Delta}$ is the stabilizer of the partition $\Delta \cup (\Omega-\Delta)$, and if $|\Delta|<n/2$ then $M_{\Delta}$ is the setwise stabilizer of $\Delta$.

Define two graphs $Q^{(1)}$, $Q^{(2)}$, which both have $G$ as set of vertices. There is an edge between $x$ and $y$ in $Q^{(1)}$ if $\langle x,y \rangle = G$, and there is an edge between $x$ and $y$ in $Q^{(2)}$ if $\langle x,y \rangle \geq A_n$. Note that if 
$\{x,y\}$ is an edge of $Q^{(i)}$, then $x$ and $y$ do not belong to the same member of $\mathscr{M}^{(i)}$. Since $\mathscr{M}^{(2)}$ is a covering of $G$, this proves that the clique number of $Q^{(2)}$ is at most $|\mathscr{M}^{(2)}|$. The third author observed in \cite{Maroti} that $\sigma(G) \leq |\mathscr{M}^{(1)}|+\sum_{i=1}^q \binom{n}{i}$, where $q := \lfloor n/3 \rfloor$, and this upper bound is asymptotically equal to $|\mathscr{M}^{(1)}|$.

We are left to prove that $|\mathscr{M}^{(i)}| \leq \omega(Q^{(i)})$ for $i=1,2$, where $\omega(Q^{(i)})$ denotes the clique number of $Q^{(i)}$, that is, the maximal number of vertices in a complete subgraph of $Q^{(i)}$.

For every $i \in \{1,2\}$ and for every $\Delta \in \mathscr{S}^{(i)}$ let
$$C(\Delta) := M_{\Delta} \cap \Pi^{(i)}.$$
Choose, uniformly and independently, an element $g_{\Delta}$ in each $C(\Delta)$, $\Delta \in \mathscr{S}^{(i)}$.

Note that the sets $C(\Delta)$ are pairwise disjoint. If $\Delta \in \mathscr{S}^{(1)}$ then a simple counting argument shows that $|C(\Delta)|= (2/n) (n/2)!^2$. If $\Delta \in \mathscr{S}^{(2)}-\mathscr{S}^{(1)}$ then $|C(\Delta)|=(|\Delta|-1)! (n-|\Delta|-1)!$. In particular, we always have 
\begin{equation}
\label{ll1}
|C(\Delta)| \geq (2/n)^2 (n/2)!^2.
\end{equation}

We define a graph $\Gamma^{(i)}$ for $i=1,2$. The vertices are the two element subsets of $\mathscr{S}^{(i)}$. Two distinct vertices $v$, $v'$ are connected by an edge if and only if $v \cap v' \neq \varnothing$. The valency of every vertex of $\Gamma^{(i)}$ is $2(|\mathscr{S}^{(i)}|-2) \leq 2^{n+1}$. For every vertex $v=\{\Delta_1,\Delta_2\}$ of $\Gamma^{(i)}$ define $E_v$ to be the event ``$\langle g_{\Delta_1},g_{\Delta_2} \rangle \neq G$'' if $i=1$, and ``$\langle g_{\Delta_1},g_{\Delta_2} \rangle \not \geq A_n$'' if $i=2$.

We will apply Theorem \ref{local} in the case of the graph $\Gamma^{(i)}$ defined above. 

Given a vertex $v$ of $\Gamma^{(i)}$, let $A$ be the set of vertices $w$ of $\Gamma^{(i)}$ with the property that $v \cap w = \varnothing$. The condition that $E_{v}$ is independent of the set of events 
$\{ E_{w} \}_{w \in A}$, mentioned in Theorem \ref{local}, means that 
$$P\Big(E_{v} \cap \bigcap_{w \in A'} \overline{E_{w}}\Big) = 
P(E_{v}) \cdot P\Big(\bigcap_{w \in A'} \overline{E_{w}}\Big),$$
for every subset $A'$  of $A$. But this is clear since $v \cap (\bigcup_{w \in A'} w) = \varnothing$ by the definition of $\Gamma^{(i)}$ and so the choices of $g_{\Delta}$ with $\Delta \in v$ are independent of the choices of $g_{\Delta}$ with $\Delta \in \cup_{w \in A'} w$.

The conclusion of Theorem \ref{local}  
is that there exists a set 
$S \subseteq G$ containing precisely one 
element $g_{\Delta}$ in every $C(\Delta)$, 
for every 
$\Delta \in \mathscr{S}^{(i)}$, 
such that 
$\langle g_{\Delta_1},g_{\Delta_2} \rangle 
= G$ for $i=1$ and $\langle g_{\Delta_1},g_{\Delta_2} \rangle 
\geq A_n$ for $i=2$, for every $g_{\Delta_1} \neq 
g_{\Delta_2}$ in $S$. 
This would imply that $$\omega(Q^{(i)}) \geq |S| 
= |\mathscr{S}^{(i)}| = |\mathscr{M}^{(i)}|$$for $i \in \{1,2\}$, which is what we need.
Note that the fact that $|S| = |\mathscr{S}^{(i)}|$ follows from the fact that the sets $C(\Delta)$ are pairwise disjoint.

We will repeatedly use Stirling's inequalities. 

\begin{lemma}
\label{Stirling}
For every positive integer $m$ we have
$$\sqrt{2 \pi m} (m/e)^m \leq m! \leq e \sqrt{m} (m/e)^m.$$
\end{lemma}




For every $H \leq G$ and $\Delta \in \mathscr{S}^{(i)}$ define
$$f_{\Delta}(H) := \frac{|C(\Delta) \cap H|}{|C(\Delta)|}.$$
Given $d,m>1$ such that $n=dm$, denote by 
$\mathcal{W}_{d,m}$ the class
of imprimitive 
maximal subgroups of $G$ isomorphic to 
$S_d \wr S_m$, that is, stabilizers of partitions of $\{1,\ldots,n\}$ consisting of $m$ blocks of size $d$ each.

\begin{lemma}\label{l:imprimitive_1}
Let $\Delta \in \mathscr{S}^{(i)}$ and $W\in\mathcal{W}_{d,m}$.
\begin{enumerate}
\item Assume $m=3$. If $\Delta \in \mathscr{S}^{(1)}$ then $f_{\Delta}(W) \neq 0$ if and only if the intersection of $\Delta$ with each block of $W$ has size $n/6$, in which case
$$f_{\Delta}(W) \leq \frac{(n/6)!^{6}}{(n/2)!^2}\cdot n^{O(1)} \leq (1/3)^n \cdot n^{O(1)}.$$
If $\Delta \in \mathscr{S}^{(2)}-\mathscr{S}^{(1)}$ then $f_{\Delta}(W) \neq 0$ only if $3$ divides $a=|\Delta|$ and the elements of $C(\Delta)$ permute transitively the $3$ blocks of $W$, moreover in this case setting $b=n-a$, we have
$$f_{\Delta}(W) \leq \frac{(a/3)!^3 \cdot (b/3)!^3}{a! \cdot b!} \cdot n^{O(1)} \leq (1/3)^n \cdot n^{O(1)}.$$
\item Assume $m=4$. If $\Delta \in \mathscr{S}^{(1)}$ then $f_{\Delta}(W) \neq 0$ if and only if $\Delta$ is a union of $2$ blocks of $W$, in which case
$$f_{\Delta}(W) \leq \frac{(n/4)!^{4}}{(n/2)!^2}\cdot n^{O(1)} \leq (1/2)^n \cdot n^{O(1)}.$$
If $\Delta \in \mathscr{S}^{(2)}-\mathscr{S}^{(1)}$ then $f_{\Delta}(W) \neq 0$ only if $a=|\Delta|=n/4$ and $\Delta$ is a block of $W$, moreover in this case setting $b=n-a=3n/4$, we have
$$f_{\Delta}(W) \leq \frac{(b/3)!^3}{b!} \cdot n^{O(1)} \leq (1/3)^{3n/4} \cdot n^{O(1)}.$$
Note that $(1/3)^{3/4} < 1/2$.
\end{enumerate}
\end{lemma}

\begin{proof}
The first inequality in each statement follows from bounding $|C(\Delta) \cap H|$ and $|C(\Delta)|$ separately, recalling that the members of $\mathscr{S}^{(2)}-\mathscr{S}^{(1)}$ have odd size. Let $k$ be any constant positive integer and $x$ a positive integer divisible by $k$. Using Lemma \ref{Stirling} we deduce that
\begin{align*}
\frac{(x/k)!^k}{x!} & \leq \frac{e^k x^{k/2} (x/(ke))^x}{(x/e)^x} = (1/k)^x \cdot e^k \cdot x^{k/2}. 
\end{align*}
The second inequality in each statement of the lemma follows from this observation. This concludes the proof.
\end{proof}

\begin{lemma}[Lemma 4 of \cite{B}] \label{mdelta}
Let $n$ be a positive integer. Let $M$ be a fixed subgroup of $G$. Let $g$ be a fixed element of $G$, and suppose that $g$ is an $n$-cycle, or that $g$ is an $(s,n-s)$-cycle for some integer $s$ such that $1 \leq s \leq n/2$. Then $g$ is contained in at most $n^2$ conjugates of $M$ in $G$.
\end{lemma}

The following lemma is a consequence of \cite[Theorem 3]{B}.

\begin{lemma} \label{bimpr}
Let $d \geq 2$, $m \geq 5$ be integers such that $n = dm$. Then $$|S_d \wr S_m| = d!^m m! \leq (n/5e)^{n} \cdot n^{O(1)}.$$
\end{lemma}

From now on $i$ will be $1$ or $2$.

Let $\mathscr{H}^{(i)}$ be the family of all maximal subgroups of $G$ outside $\mathscr{M}^{(i)}$. Write $\mathscr{H}^{(i)} = \bigcup_{j=1}^5 \mathscr{H}_j$ where $\mathscr{H}_1$ is the family of intransitive maximal subgroups of $G$ outside $\mathscr{M}^{(i)}$, $\mathscr{H}_2$ is the family of primitive maximal subgroups of $G$, $\mathscr{H}_j$ is the family of imprimitive maximal subgroups of $G$ with $j$ blocks for $j \in \{3,4\}$ and $\mathscr{H}_5$ is the family of imprimitive maximal subgroups of $G$ with at least $5$ blocks. Let $J:=\{1,2,3,4,5\}$. For $j \in J$ and $v=\{\Delta_1,\Delta_2\} \in V(\Gamma^{(i)})$, let $E_v^j$ be the event ``$g_{\Delta_1}, g_{\Delta_2}$ both belong to some $H \in \mathscr{H}_j$'', so that $P(E_v) \leq \sum_{j \in J} P(E_v^j)$. 

We will prove that 
$$ \sum_{j \in J} P(E_v^j) \leq 
\frac{1}{2^{n+3}},$$
which for sufficiently large $n$ is smaller than 
$\frac{1}{e(d+1)}$.

Let $[H]$ denote the 
$G$-conjugacy class of a subgroup $H$ of $G$. 
Let $m_{\Delta}([H])$ be the number of different 
conjugates of $H$ that contain a fixed element 
of $ C(\Delta)$. By Lemma \ref{mdelta}, $m_{\Delta}([H]) \leq n^2$ always.

If $H \in 
\mathscr{M}^{(i)}$ then at least one of 
$f_{\Delta_1}(H)$ and $f_{\Delta_2}(H)$ is $0$ for $\Delta_{1} \not= \Delta_{2}$. 
Therefore in the computation of $P(E_v)$ we 
restrict our attention to the maximal subgroups 
of $G$ outside $\mathscr{M}^{(i)}$. 

In the following sum we let $[H]$ vary in the set of conjugacy 
classes of elements of $\mathscr{H}_j$ with $j \in J$. We have
\begin{align*}
P(E_v^j) & \leq \sum_{[H]} \sum_{K \in [H]} 
f_{\Delta_1}(K) \cdot f_{\Delta_2}(K) \\
& = \sum_{[H]} \sum_{K \in 
[H]}\frac{\abs{C(\Delta_1)
\cap K}}{\abs{C(\Delta_1)}} f_{\Delta_2}(K) \\
& = \sum_{[H]} \sum_{K \in [H]}\sum_{g\in 
C(\Delta_1)\cap K}
\frac{1}{\abs{C(\Delta_1)}}f_{\Delta_2}(K) \\
& = \sum_{[H]} \sum_{g \in C(\Delta_1)} 
\frac{1}{\abs{C(\Delta_1)}}\sum_{K \in [H] \atop 
g\in K}f_{\Delta_2}(K) \\
& \leq \sum_{[H]} \sum_{g \in 
C(\Delta_1)}\frac{1}{\abs{C(\Delta_1)}} \cdot 
m_{\Delta_1}([H]) \cdot \max_{K \in [H]} 
f_{\Delta_2}(K) \\
& = \sum_{[H]} m_{\Delta_1}([H])\cdot \max_{K \in 
[H]} f_{\Delta_2}(K).
\end{align*}

For $v = \{ \Delta_{1}, \Delta_{2} \}$ let $c_{v,j}$ be the number of conjugacy classes 
of subgroups in $\mathscr{H}_j$ such that there 
exists $H$ in such a class such that $H \cap 
C(\Delta_1) \neq \varnothing$ and $H \cap 
C(\Delta_2) \neq \varnothing$. We deduce that
\begin{align} \label{inequality_argument_1}
P(E_v^j) & \leq c_{v,j} \cdot 
\min_{\{i_1,i_2\}=\{1,2\}} \left( \max_{H \in 
\mathscr{H}_j \atop K \in [H]} \left( 
m_{\Delta_{i_1}}([H])\cdot f_{\Delta_{i_2}}(K) 
\right) \right).
\end{align}

For $v = \{ \Delta_{1}, \Delta_{2} \}$ denote by $s_{v,j}$ the number of members of 
$\mathscr{H}_j$ intersecting both $C(\Delta_1)$ 
and $C(\Delta_2)$ non-trivially. Then
\begin{align} \label{inequality_argument_2}
P(E_v^j) & \leq \sum_{H \in \mathscr{H}_j} 
f_{\Delta_1}(H) \cdot f_{\Delta_2}(H) \leq s_{v,j}
\cdot \max_{H \in \mathscr{H}_j} \left( 
f_{\Delta_1}(H) \cdot f_{\Delta_2}(H) \right).
\end{align}

We will use inequality 
(\ref{inequality_argument_1}) if $j \neq 4$ or $(|\Delta_1|,|\Delta_2|) \neq (n/2,n/2)$ and 
we will use inequality
(\ref{inequality_argument_2}) if $j=4$ and $(|\Delta_1|,|\Delta_2|) = (n/2,n/2)$.

\begin{lemma}
\label{cs}
Let $v=\{\Delta_1,\Delta_2\}$ be a vertex of 
$\Gamma^{(i)}$ and let $j \in J$. Then 
$c_{v,2} \leq n$ (for $n$ large enough) and 
$c_{v,j} \leq 1$ for $j \in \{3,4\}$. 
Moreover $c_{v,5} \leq 2 \sqrt{n}$. If $(|\Delta_1|,|\Delta_2|)=(n/2,n/2)$ then $s_{v,4} \leq 1$.
\end{lemma}
\begin{proof}
By \cite{LS}, $c_{v,2}
\leq n$ for $n$ large enough. We remark that this is the only point where we use CFSG. $c_{v,5}$ 
is at most the number of positive 
divisors of $n$ less than
$n$, and this is at most $2 \sqrt{n}$. If 
$j \in \{3,4\}$, then 
$\mathscr{H}_j$ is a single conjugacy class of 
subgroups of $G$, therefore $c_{v,j} \leq 1$.

Assume now that $(|\Delta_1|,|\Delta_2|)=(n/2,n/2)$, so that both $C(\Delta_1)$ and 
$C(\Delta_2)$ consist of $n$-cycles. We will prove that $s_{v,4} \leq 1$. Let $H$ 
be an imprimitive maximal subgroup of $G$ 
isomorphic to $S_{n/4} \wr S_4$. The two sets $H 
\cap C(\Delta_1)$ and $H \cap C(\Delta_2)$ are 
both non-empty only if the four imprimitivity 
blocks of $H$ are exactly: $\Delta_1 \cap 
\Delta_2$, $\Delta_1-\Delta_2$, 
$\Delta_2-\Delta_1$ and $\Omega-(\Delta_1 \cup 
\Delta_2)$, in which case $|\Delta_1 \cap 
\Delta_2| = n/4$. This implies the result.
\end{proof}

We will bound each $P(E_v^j)$. 
Note that $P(E_v^1)=0$ in all cases, since no intransitive subgroup contains $n$-cycles and the only intransitive maximal subgroups containing elements of cycle type $(a,b)$ with $a,b$ odd are the ones belonging to $\mathscr{M}^{(2)}$.

\begin{enumerate}
\item[(i)] $j=2$. By a result of Praeger and Saxl \cite{PS}, the order of any member of $\mathscr{H}_2$ is at most $4^n$. By Lemmas \ref{mdelta}, \ref{cs}, inequalities (\ref{ll1}) and (\ref{inequality_argument_1}), 
\begin{align*}
P(E_v^2) & \leq n^{3} \cdot \max_{H \in 
\mathscr{H}_2 \atop K \in [H]} f_{\Delta_2}(K) \leq n^3 \cdot \frac{4^n}{(2/n)^2 (n/2)!^2}.
\end{align*}

\item[(ii)] $j=3$. By Lemmas \ref{l:imprimitive_1}, \ref{mdelta}, \ref{cs} and inequality (\ref{inequality_argument_1}), 
$$P(E_v^3) \leq n^{2} \cdot \max_{ H \in 
\mathscr{H}_3 \atop K \in [H] } f_{\Delta_2}(K) \leq (1/3)^n \cdot n^{O(1)}.$$

\item[(iii)] $j=4$. If $(|\Delta_1|,|\Delta_2|) \neq (n/2,n/2)$ then, without loss of generality, we may assume $|\Delta_2| \neq n/2$. In this case, by Lemmas \ref{l:imprimitive_1}, \ref{mdelta}, \ref{cs} and inequality (\ref{inequality_argument_1}), 
$$P(E_v^4) \leq  n^{2} \cdot \max_{ H \in 
\mathscr{H}_4 \atop K \in [H] } f_{\Delta_2}(K)   \leq (1/3)^{3n/4} \cdot n^{O(1)}.$$
Assume now that $(|\Delta_1|,|\Delta_2|)=(n/2,n/2)$. Since $s_{v,4} \leq 1$ by Lemma \ref{cs}, we have
$$P(E_v^4) \leq \max_{ H \in 
\mathscr{H}_4}  (f_{\Delta_1}(H) \cdot f_{\Delta_2}(H)) \leq (1/4)^n \cdot n^{O(1)},$$ by inequality 
(\ref{inequality_argument_2}) and Lemma \ref{l:imprimitive_1}.

\item[(iv)] $j=5$. Fix $H \in \mathscr{H}_5$.  
Then, by Lemma \ref{bimpr}, inequality (\ref{ll1}) and Lemma \ref{Stirling},
\begin{align*}
f_{\Delta}(H) & \leq \frac{|H|}{|C(\Delta)|} \leq  \frac{(n/5e)^n \cdot n^{O(1)}}{(n/2)!^2} \leq \frac{(n/5e)^n \cdot n^{O(1)}}{(n/2e)^n} = (2/5)^n \cdot n^{O(1)}.
\end{align*}
The set $\mathscr{H}_5$ contains at most $2 \sqrt{n}$ 
classes of subgroups. For every $H \in \mathscr{H}_{5}$, we have $m_{\Delta_1}([H]) \leq n^2$ by Lemma \ref{mdelta}, hence
$$P(E_v^5) \leq 2 \sqrt{n} \cdot n^{2}
\cdot \max_{  H \in 
\mathscr{H}_5 \atop K \in [H]   } f_{\Delta_2}(K) \leq (2/5)^n \cdot n^{O(1)},$$ 
by inequality (\ref{inequality_argument_1}).

\end{enumerate}

Combining everything, we deduce that $$P(E_v) \leq \sum_{j \in J} P(E_v^j) \leq (1/3)^{3n/4} \cdot n^{O(1)},$$
which is smaller than $2^{-n-3}$ for every large enough $n$.

\section{Proof of Proposition \ref{CFSG}}

Eberhard and Virchow \cite[Theorem 1.1]{EV} proved, without CFSG, that for every $\epsilon > 0$ the probability $p(n)$ that a random pair of elements from $S_n$ generates $S_n$ or $A_n$ is $$1 - \frac{1}{n} + \Theta(n^{-2+\epsilon}),$$ for every sufficiently large $n$. The same asymptotic formula holds \cite[Theorem 1.2]{EV} for the probability $a(n)$ that a random pair of elements from $A_n$ generates $A_n$. Let $b(n)$ be the probability that a random pair of elements from $S_{n} \setminus A_n$ generates $S_n$. Let $c(n)$ be the probability that a random element from $A_n$ and a random element from $S_{n} \setminus A_{n}$ generate $S_n$. Observe that $b(n) = c(n)$ since $\langle x, y \rangle = \langle x y^{-1}, y \rangle$ where $x$ and $y$ are in $S_{n} \setminus A_n$. Since $$p(n) = \frac{a(n) + b(n) + 2 c(n)}{4},$$ it follows that $b(n) = (4p(n) - a(n))/3$.
Fix $\epsilon > 0$. We have universal positive constants $c_{1}$ and $c_{2}$ by \cite{EV} such that $1 - n^{-1} - c_{1}n^{-2+\epsilon}$ is smaller than both $p(n)$ and $a(n)$ and $a(n) < 1 - n^{-1} + c_{2}n^{-2+\epsilon}$ . Thus $1 - n^{-1} - (1/3)(4c_{1}+c_{2}) n^{-2 + \epsilon} < b(n)$.

Following Liebeck and Shalev \cite{LS96}, define graphs $A(n)$ and $B(n)$ with vertex sets $A_n$ and $S_{n} \setminus A_{n}$ respectively such that there is an edge between vertices $x$ and $y$ if and only if $x$ and $y$ generate $A_n$ in the first case and $S_n$ in the second case. The largest size of a complete subgraph in $A(n)$ is $\omega(A_{n})$ and the largest size of a complete subgraph in $B(n)$ is at most $\omega(S_{n})$.

Tur\'an's \cite{Turan} theorem states that a simple graph on $m$ vertices which does not contain a complete subgraph of size $r+1$ has at most $(1 - \frac{1}{r})\frac{m^{2}}{2}$ edges. We apply this theorem to the graphs $A(n)$ and $B(n)$ with $m = n!/2$ vertices. Consider the graph $A(n)$. (The argument for the case of $B(n)$ is the same.) Let $r:=\omega(A_n)$. Since $A_{n}$ is not a cyclic group, observe that $A(n)$ has more than $(1 - \frac{1}{n} - c_{1}n^{-2+\epsilon})\frac{m^{2}}{2}$ edges. It follows that $$\Big(1 - \frac{1}{n} - c_{1}n^{-2+\epsilon}\Big)\frac{m^{2}}{2} < \Big(1 - \frac{1}{r}\Big)\frac{m^{2}}{2},$$ giving $r > n - c_{1} n^{\epsilon} = (1-o(1))n$. 

\section{Acknowledgements}

We would like to thank the referees for helpful comments, in particular, for improving the statement and proof of Proposition \ref{CFSG}.

\end{document}